\documentclass{proc-l}

\usepackage{xurl}
\usepackage[hidelinks]{hyperref}
\numberwithin{equation}{section}

\newtheorem{theorem}{Theorem}[section]
\newtheorem{proposition}[theorem]{Proposition}
\newtheorem{lemma}[theorem]{Lemma}
\newtheorem{corollary}[theorem]{Corollary}

\theoremstyle{definition}
\newtheorem{definition}[theorem]{Definition}

\theoremstyle{remark}
\newtheorem{remark}[theorem]{Remark}

\newcommand{\U}{\mathbb U}
\newcommand{\QU}{\mathbb Q\mathbb U}
\newcommand{\N}{\mathbb N}
\newcommand{\R}{\mathbb R}
\newcommand{\one}{\mathbf 1}
\newcommand{\half}{\mathbf h}
\newcommand{\ext}{\operatorname{ext}}
\newcommand{\conv}{\operatorname{conv}}
\newcommand{\cl}{\operatorname{cl}}
\newcommand{\diam}{\operatorname{diam}}

\newcommand{\rank}{\operatorname{rank}}

\newcommand{\bary}{\operatorname{bar}}
\newcommand{\dd}{\mathrm d}

\begin{document}

\title[Finite and Urysohn obstructions to S-prime]{Finite and Urysohn obstructions to Sabok's S-prime simplex questions}

\author{Yutong Zhang and Yaoran Yang}
\address{School of Mathematics, Sichuan University, Chengdu, 610065, China}
\email{yutongzhang@stu.scu.edu.cn; yangyaoran@stu.scu.edu.cn}

\date{}
\commby{[Editor Name]}

\begin{abstract}
Sabok asked whether the compact convex set \(S'(X)\) attached to a separable
metric space of diameter at most one is always a simplex, and whether
\(S'(\mathbb U_1)\) is the Poulsen simplex.  We give negative answers.  For
finite \(X=\{x_1,\ldots,x_m\}\), \(S'(X)\) is affinely homeomorphic to the
convex hull of the rows \(r_i=(d(x_i,x_1),\ldots,d(x_i,x_m))\) of the distance
matrix; it is a simplex exactly when these rows are affinely independent.  The
diameter-one four-cycle gives the minimal finite obstruction.  For the Urysohn
sphere, using the rational Urysohn sphere \(D\) as coordinates, we identify the
coordinate model \(S'_D(\mathbb U_1)\) with the Kat\'etov compactum \(K(D)\).
Four explicit extreme points \(f_A,g_A,\mathbf 1,\mathbf h\) satisfy
\(f_A+g_A=\mathbf 1+\mathbf h\), giving two distinct representing measures for
\((3/4)\mathbf 1\).  Hence \(S'(\mathbb U_1)\) is not a Choquet simplex.
\end{abstract}

\maketitle

\section{Introduction}

Sabok introduced, for a separable metric space \(X\) of diameter at most one, a
compact convex set denoted \(S'(X)\) in his work on the complexity of affine
homeomorphism of Choquet simplices.  In the distance-coordinate form used here,
one chooses a countable dense set \(D\subseteq X\) and takes the closed convex
hull of
\[
        x\longmapsto (d(x,d))_{d\in D}
\]
in a Hilbert cube; Sabok's coordinate-set invariance identifies the compacta
obtained from different dense coordinate sets \cite[Proposition~5.11]{Sabok2016}.
Sabok asked \cite[Question~9.4]{Sabok2016} whether \(S'(X)\) is always a
simplex, and whether \(S'(\mathbb U_1)\) is affinely homeomorphic to the Poulsen
simplex, the unique nontrivial metrizable Choquet simplex with dense extreme
boundary \cite{Poulsen1961,LindenstraussOlsenSternfeld1978}.  We give negative
answers to both questions.  This concerns Sabok's projected compactum \(S'\),
not the unfolded compactum \(S(\mathbb U_1,d_{\mathbb U_1})\), which Sabok
proved to be a simplex \cite[Proposition~5.19]{Sabok2016}.

The finite answer is already visible in the distance matrix.  If
\(X=\{x_1,\ldots,x_m\}\) is finite, then \(S'(X)\) is affinely homeomorphic to
\[
        P_X=\conv\{r_1,\ldots,r_m\},
        \qquad
        r_i=(d(x_i,x_1),\ldots,d(x_i,x_m)).
\]
Consequently \(S'(X)\) is a simplex exactly when the distance rows are affinely
independent, equivalently
\[
        \ker M_X\cap\{a\in\R^m:\sum_i a_i=0\}=\{0\}.
\]
The diameter-one four-cycle satisfies \(r_0+r_2=r_1+r_3\), giving a
non-degenerate parallelogram and hence a minimal finite counterexample; no
counterexample exists on fewer than four points, and examples exist in every
finite cardinality at least four.

The finite relation does not prove the Urysohn statement.  If the same four
points are embedded in \(\mathbb U_1\), the additional distance coordinates to a
dense subset can destroy the finite dependence; indeed Sabok proved a finite
independence property for Urysohn distance coordinates
\cite[Proposition~5.15]{Sabok2016}.  We therefore work in the full
coordinate compactum.  Taking \(D=\mathbb Q\mathbb U_1\subset\mathbb U_1\) as the
coordinate set, we identify \(S'_D(\mathbb U_1)\) with
\[
        K(D)=\{f\in[0,1]^D: |f(x)-f(y)|\le d(x,y)\le f(x)+f(y)
        \text{ for all }x,y\in D\},
\]
the compactum of bounded Kat\v{e}tov functions on the rational Urysohn sphere.
For a side-one equilateral triple \(A=\{a_0,a_1,a_2\}\subset D\), put
\(r_A(x)=d(x,A)\) and
\[
        f_A(x)=\min\{1,\tfrac12+r_A(x)\},\qquad
        g_A(x)=\max\{\tfrac12,1-r_A(x)\}.
\]
We prove that \(f_A\), \(g_A\), the constant function \(\one\), and the constant
function \(\half\equiv1/2\) are extreme points of \(K(D)\), while
\[
        f_A+g_A=\one+\half.
\]
Thus the point \((3/4)\one\) has two different representing probability measures
carried by extreme points.  Hence \(S'(\mathbb U_1)\) is not a Choquet simplex
and cannot be the Poulsen simplex.

Sections~\ref{sec:finite-model}--\ref{sec:barycentric-finite} treat finite
metric spaces.  Sections~\ref{sec:urysohn-model}--\ref{sec:urysohn-non-simplex}
prove the Urysohn obstruction.

\section{The finite-coordinate model and criterion}
\label{sec:finite-model}

We first fix the finite notation.  Let \((X,d)\) be a metric space with
\(\diam(X)\le1\).  If \(D=(d_n)_{n\ge1}\) is a countable dense sequence, write
\[
        \iota_D:X\longrightarrow[0,1]^\N,
        \qquad
        \iota_D(x)=\bigl(d(x,d_1),d(x,d_2),\ldots\bigr).
\]
For this chosen distance-coordinate set, write
\[
        S'_D(X)=\overline{\conv}^{[0,1]^\N}\,\iota_D(X).
\]
This is the distance-coordinate instance of Sabok's construction.  When no
ambiguity can arise, and in view of Sabok's affine invariance under changes of
dense coordinate set, the subscript \(D\) is suppressed.  For finite
spaces, repetitions in the dense sequence only repeat coordinates; deleting
repeated coordinates gives an affine homeomorphic compact convex set.

\begin{definition}
Let \(X=\{x_1,\ldots,x_m\}\) be a finite metric space of diameter at most one.
Its distance matrix and distance rows are
\[
        M_X=(d(x_i,x_j))_{1\le i,j\le m},
        \qquad
        r_i=(d(x_i,x_1),\ldots,d(x_i,x_m))\in\R^m.
\]
The distance-row polytope is
\[
        P_X=\conv\{r_1,\ldots,r_m\}\subseteq\R^m.
\]
\end{definition}

\begin{proposition}[finite reduction]\label{prop:finite-reduction}
Let \(X=\{x_1,\ldots,x_m\}\) be finite of diameter at most one.  Then \(S'(X)\)
is affinely homeomorphic to \(P_X\).  More precisely, for any dense enumeration
of \(X\), deleting repeated coordinates and then, if necessary, applying the
coordinate permutation induced by the first occurrences identifies \(S'(X)\)
with
\[
        \conv\{r_1,\ldots,r_m\}=P_X.
\]
\end{proposition}

\begin{proof}
Let \((u_n)_{n\ge1}\) be a dense enumeration of \(X\), and let
\(u_{n_k}=x_{\sigma(k)}\), \(1\le k\le m\), be the first occurrences of the
points of \(X\).  Projection to the coordinates \(n_1,\ldots,n_m\) sends
\(\iota_u(x_i)\) to
\[
        \bigl(d(x_i,x_{\sigma(1)}),\ldots,d(x_i,x_{\sigma(m)})\bigr),
\]
a coordinate permutation of \(r_i\).  Conversely, each deleted coordinate is a
repeat of one of the retained coordinates on all distance rows and hence on
their convex hull, so re-inserting the repeated coordinates gives the affine
inverse.  Since there are only finitely many rows, the convex hull is compact;
a final coordinate permutation identifies it with \(P_X\).
\end{proof}

\begin{lemma}[distance rows are vertices]\label{lem:rows-vertices}
For every finite metric space \(X=\{x_1,\ldots,x_m\}\) of diameter at most one,
each row \(r_i\) is an extreme point of \(P_X\).  Thus
\[
        \ext(P_X)=\{r_1,\ldots,r_m\}.
\]
\end{lemma}

\begin{proof}
If, for some fixed \(i\), one had
\(r_i=\sum_{j\ne i}\alpha_jr_j\), \(\alpha_j\ge0\), \(\sum_{j\ne i}\alpha_j=1\),
then the \(i\)-th coordinate would give
\[
        0=r_i(i)=\sum_{j\ne i}\alpha_jd(x_j,x_i).
\]
All summands are nonnegative and \(d(x_j,x_i)>0\) for \(j\ne i\), so every
\(\alpha_j\) is zero, a contradiction.  Thus every \(r_i\) is a vertex, and the
listed vertices are exactly the extreme points of \(P_X\).
\end{proof}

For the compact metrizable convex sets considered here, we use the standard
Choquet terminology: a compact convex set \(K\) is a simplex if each point of
\(K\) has a unique representing probability measure carried by \(\ext(K)\);
see, for example, Alfsen's monograph \cite{Alfsen1971}.  In a finite-dimensional
polytope this is equivalent to ordinary affine independence of the vertices.

\begin{theorem}[finite criterion]\label{thm:finite-criterion}
Let \(X=\{x_1,\ldots,x_m\}\) be a finite metric space of diameter at most one,
and let \(M_X=(d(x_i,x_j))_{i,j}\).  The following are equivalent:
\begin{enumerate}
\item[\textup{(i)}] \(S'(X)\) is a Choquet simplex;
\item[\textup{(ii)}] \(P_X\) is an \((m-1)\)-simplex;
\item[\textup{(iii)}] the rows \(r_1,\ldots,r_m\) of \(M_X\) are affinely independent;
\item[\textup{(iv)}]
\[
        \rank\begin{pmatrix}
        1&d(x_1,x_1)&\cdots&d(x_1,x_m)\\
        1&d(x_2,x_1)&\cdots&d(x_2,x_m)\\
        \vdots&\vdots&&\vdots\\
        1&d(x_m,x_1)&\cdots&d(x_m,x_m)
        \end{pmatrix}=m;
\]
\item[\textup{(v)}]
\[
        \ker M_X\cap H_0=\{0\},
        \qquad
        H_0=\{a\in\R^m:\langle a,\one\rangle=0\}.
\]
\end{enumerate}
\end{theorem}

\begin{proof}
By Proposition~\ref{prop:finite-reduction}, \(S'(X)\) is affinely homeomorphic to \(P_X\).
By Lemma~\ref{lem:rows-vertices}, the extreme points of \(P_X\) are exactly
\(r_1,\ldots,r_m\).  A finite-dimensional polytope with vertex set \(V\) is a
Choquet simplex exactly when \(V\) is affinely independent: if \(V\) is
affinely independent, barycentric coordinates are unique, while if \(V\) has a
nontrivial affine dependence, separating the positive and negative coefficients
gives two different probability measures on \(V\) with the same barycenter.
Thus \textup{(i)}, \textup{(ii)}, and \textup{(iii)} are equivalent.

The equivalence between \textup{(iii)} and \textup{(iv)} is the standard
augmented-matrix criterion for affine independence.  Indeed, an affine
dependence among the rows is a vector \(a=(a_i)\in\R^m\) satisfying
\[
        \sum_{i=1}^m a_i=0,
        \qquad
        \sum_{i=1}^m a_i r_i=0.
\]
Writing the rows as a matrix gives precisely
\[
        \langle a,\one\rangle=0,
        \qquad
        a^T M_X=0.
\]
Since \(M_X\) is symmetric, \(a^TM_X=0\) is equivalent to \(M_Xa=0\).  Hence
\textup{(iii)} is equivalent to
\[
        \{a\in\R^m:a^TM_X=0,\ \langle a,\one\rangle=0\}=\{0\},
\]
which is \textup{(v)}.
\end{proof}

\section{The four-point obstruction and finite sharpness}
\label{sec:c4}

Let
\[
        C_4=\mathbb Z/4\mathbb Z
\]
with the diameter-one cycle metric
\[
        d(i,j)=\frac12\min\{k(i,j),4-k(i,j)\},
        \qquad i,j\in\mathbb Z/4\mathbb Z,
\]
where \(k(i,j)\in\{0,1,2,3\}\) is the representative of \(i-j\pmod 4\).
Thus adjacent points have distance \(1/2\), opposite points have distance \(1\),
and \(\diam(C_4)=1\).  The rows are
\begin{align*}
        r_0&=(0,\tfrac12,1,\tfrac12),&
        r_1&=(\tfrac12,0,\tfrac12,1),\\
        r_2&=(1,\tfrac12,0,\tfrac12),&
        r_3&=(\tfrac12,1,\tfrac12,0).
\end{align*}
They satisfy the affine relation
\begin{equation}\label{eq:c4-relation}
        r_0+r_2=r_1+r_3=(1,1,1,1).
\end{equation}
Equivalently, for the distance matrix
\[
        M_{C_4}=\begin{pmatrix}
        0&\tfrac12&1&\tfrac12\\
        \tfrac12&0&\tfrac12&1\\
        1&\tfrac12&0&\tfrac12\\
        \tfrac12&1&\tfrac12&0
        \end{pmatrix},
        \qquad
        \eta=(1,-1,1,-1)^T,
\]
one has
\begin{equation}\label{eq:c4-kernel}
        M_{C_4}\eta=0,
        \qquad
        \langle\eta,\one\rangle=0.
\end{equation}

\begin{theorem}[finite counterexample]\label{thm:c4-counterexample}
For the four-point metric space \(C_4\) above, \(S'(C_4)\) is not a Choquet
simplex.  More precisely,
\[
        S'(C_4)\cong P_{C_4}=\conv\{r_0,r_1,r_2,r_3\}
\]
is a non-degenerate parallelogram.  Its center
\[
        b=(\tfrac12,\tfrac12,\tfrac12,\tfrac12)
\]
has two distinct representing probability measures carried by
\(\ext(P_{C_4})\):
\[
        \mu_{02}=\frac12\delta_{r_0}+\frac12\delta_{r_2},
        \qquad
        \mu_{13}=\frac12\delta_{r_1}+\frac12\delta_{r_3},
\]
and
\[
        b=\int_{P_{C_4}} z\,\dd\mu_{02}(z)
         =\int_{P_{C_4}} z\,\dd\mu_{13}(z),
        \qquad
        \mu_{02}\ne\mu_{13}.
\]
\end{theorem}

\begin{proof}
Proposition~\ref{prop:finite-reduction} identifies \(S'(C_4)\) with \(P_{C_4}\), and
Lemma~\ref{lem:rows-vertices} gives
\[
        \ext(P_{C_4})=\{r_0,r_1,r_2,r_3\}.
\]
The relation \eqref{eq:c4-relation} gives
\[
        \frac12r_0+\frac12r_2
        =\frac12r_1+\frac12r_3
        =(\tfrac12,\tfrac12,\tfrac12,\tfrac12)=b.
\]
Moreover,
\[
        r_1-r_0=(\tfrac12,-\tfrac12,-\tfrac12,\tfrac12),
        \qquad
        r_2-r_0=(1,0,-1,0),
\]
are linearly independent.  Thus the four vertices are not collinear.  Since the
midpoints of the two diagonals agree by \eqref{eq:c4-relation}, their convex
hull is a non-degenerate parallelogram.  The two convex decompositions use
disjoint pairs of vertices, hence the corresponding measures \(\mu_{02}\) and
\(\mu_{13}\) are distinct.  Since a Choquet simplex has unique representing
measures carried by its extreme boundary, \(P_{C_4}\), and therefore
\(S'(C_4)\), is not a simplex.
\end{proof}

\begin{corollary}\label{cor:q1-finite}
There is a separable metric space \(X\) of diameter one for which \(S'(X)\) is
not a simplex.  Hence the assertion that \(S'(X)\) is always a simplex for
separable metric spaces \(X\) is false.
\end{corollary}

\begin{proof}
The finite space \(C_4\) is separable and has diameter one.  Apply
Theorem~\ref{thm:c4-counterexample}.
\end{proof}

The obstruction starts at four points.

\begin{proposition}[cardinalities one, two, and three]\label{prop:small-cardinalities}
If \(X\) is a finite metric space of diameter at most one and \(|X|\le3\), then
\(S'(X)\) is a simplex.
\end{proposition}

\begin{proof}
The cases \(|X|=1\) and \(|X|=2\) are immediate from Theorem~\ref{thm:finite-criterion}.
For \(|X|=3\), write
\(a=d(x_1,x_2)\), \(b=d(x_1,x_3)\), and \(c=d(x_2,x_3)\), all positive.  The
row matrix is
\[
        \begin{pmatrix}0&a&b\\ a&0&c\\ b&c&0\end{pmatrix},
\]
whose determinant is \(2abc>0\).  Hence the rows are linearly, and therefore
affinely, independent.  Theorem~\ref{thm:finite-criterion} applies.
\end{proof}

\begin{proposition}[obstructions in all larger cardinalities]\label{prop:all-larger}
For every integer \(m\ge4\), there is an \(m\)-point metric space \(X_m\) of
diameter one such that \(S'(X_m)\) is not a simplex.
\end{proposition}

\begin{proof}
For \(m=4\), take \(X_4=C_4\).  For \(m>4\), add points \(y_5,\ldots,y_m\), each at
distance \(1\) from every other new point and from every point of \(C_4\), while
retaining the metric on \(C_4\).  The triangle inequality is immediate because
all newly introduced nonzero distances are \(1\) and all old distances are at
most \(1\).  In the enlarged distance rows, the new coordinates of
\(r_0^{(m)},r_1^{(m)},r_2^{(m)},r_3^{(m)}\) are all equal to \(1\), so the old
relation \(r_0^{(m)}+r_2^{(m)}=r_1^{(m)}+r_3^{(m)}\) remains valid.  Apply
Theorem~\ref{thm:finite-criterion}.
\end{proof}

Combining Propositions~\ref{prop:small-cardinalities} and~\ref{prop:all-larger} gives the exact
finite threshold.

\begin{corollary}[finite threshold]\label{cor:threshold}
The least cardinality of a finite metric space \(X\) of diameter at most one
for which \(S'(X)\) is not a simplex is \(4\).  Moreover, non-simplex examples
exist in every finite cardinality \(m\ge4\).
\end{corollary}

\section{Barycentric reformulation and finite shadows}
\label{sec:barycentric-finite}

The finite criterion can be restated as injectivity of a barycenter map.  Let
\[
        \Delta_{m-1}=\left\{p=(p_1,\ldots,p_m)\in\R^m:
        p_i\ge0,\ \sum_{i=1}^m p_i=1\right\}
\]
be the standard probability simplex.  Define
\[
        \beta_X:\Delta_{m-1}\longrightarrow P_X,
        \qquad
        \beta_X(p)=\sum_{i=1}^m p_i r_i=p^T M_X.
\]
If \(p,q\in\Delta_{m-1}\), then
\[
        \beta_X(p)=\beta_X(q)
        \quad\Longrightarrow\quad
        (p-q)^T M_X=0,
        \qquad
        \langle p-q,\one\rangle=0.
\]
Since \(M_X\) is symmetric, the first equality is equivalent to
\(M_X(p-q)=0\).  Conversely, suppose \(0\ne a\in\ker M_X\cap H_0\).  Write
\(a=a^+-a^-\), where \(a_i^+=\max\{a_i,0\}\) and
\(a_i^-=\max\{-a_i,0\}\).  Since \(\langle a,\one\rangle=0\) and \(a\ne0\),
\[
        t=\sum_i a_i^+=\sum_i a_i^->0.
\]
Then \(p=a^+/t\) and \(q=a^-/t\) are distinct elements of \(\Delta_{m-1}\).  Since
\(M_Xa=0\) and \(M_X\) is symmetric,
\[
        \beta_X(p)-\beta_X(q)=\frac1t a^T M_X=0.
\]
Thus
\begin{equation}\label{eq:bary-kernel}
        \beta_X\text{ is injective}
        \quad\Longleftrightarrow\quad
        \ker M_X\cap H_0=\{0\}.
\end{equation}
This is exactly Theorem~\ref{thm:finite-criterion}\textup{(v)}.

For \(C_4\), the two probability vectors
\[
        p_{02}=\left(\tfrac12,0,\tfrac12,0\right),
        \qquad
        p_{13}=\left(0,\tfrac12,0,\tfrac12\right)
\]
satisfy
\[
        p_{02}-p_{13}=\frac12(1,-1,1,-1)=\frac12\eta,
        \qquad
        M_{C_4}(p_{02}-p_{13})=0,
\]
so \(\beta_{C_4}(p_{02})=\beta_{C_4}(p_{13})\).  Coordinatewise,
\[
\begin{aligned}
        \beta_{C_4}(p_{02})
        &=\frac12(0,\tfrac12,1,\tfrac12)+\frac12(1,\tfrac12,0,\tfrac12)\\
        &=(\tfrac12,\tfrac12,\tfrac12,\tfrac12)\\
        &=\frac12(\tfrac12,0,\tfrac12,1)+\frac12(\tfrac12,1,\tfrac12,0)
        =\beta_{C_4}(p_{13}).
\end{aligned}
\]
The failure of the simplex property is precisely the non-injectivity of the
finite distance-potential map
\[
        p\longmapsto\left(\sum_i p_i d(x_i,x_1),\ldots,
        \sum_i p_i d(x_i,x_m)\right).
\]

This finite obstruction does not by itself prove the Urysohn statement.  In
the finite example \(X=C_4\), the coordinate family consists only of distances
to the four points of \(X\).  If the same four points are embedded in
\(\mathbb U_1\) and one takes distances to a dense set
\(D\subseteq\mathbb U_1\), additional coordinates generally destroy the
relation \(r_0+r_2=r_1+r_3\).  Moreover, Sabok proves that for every countable
dense \(D\subseteq\mathbb U_1\), the family \(d_{\mathbb U_1}D\) is independent
in the following finite sense: for every finite set, equivalently every finite
list of pairwise distinct points,
\(x_1,\ldots,x_n\in\mathbb U_1\), one can find distance coordinates whose
\(n\times n\) evaluation matrix is invertible \cite[Proposition 5.15]{Sabok2016}.
Thus the Urysohn part requires an argument in the full compactum, not merely a
finite affine dependence among points of \(\mathbb U_1\).

\section[The Katetov model of S-prime of the Urysohn sphere]{The Katetov model of S-prime of the Urysohn sphere}
\label{sec:urysohn-model}

We use the following standard form of the rational Urysohn sphere.  It is the
Fra\"iss\'e limit of finite rational metric spaces of diameter at most one.
Equivalently, it is a countable metric space \(D=(D,d)\) with rational
distances, diameter one, and the following extension property:

\begin{quote}
For every finite \(F\subset D\) and every map \(p:F\to\mathbb Q\cap[0,1]\)
such that
\[
        |p(x)-p(y)|\le d(x,y)\le p(x)+p(y)\qquad(x,y\in F),
\]
there exists \(z\in D\) such that \(d(z,x)=p(x)\) for all \(x\in F\).
\end{quote}

Its completion is the Urysohn sphere \(\U_1\).  This is the bounded counterpart
of Urysohn's universal homogeneous metric space \cite{Urysohn1927}; the
Kat\v{e}tov construction gives a flexible modern formulation
\cite{Katetov1988,Gromov2007}.  We shall use only the exact finite extension
property displayed above.  Standard references for the descriptive and
topological-dynamical use of the Urysohn space include
\cite{GaoKechris2003,Melleray2016}.

\begin{definition}
Let \(D=\QU_1\).  Define
\[
\begin{aligned}
        K(D)=\{f\in[0,1]^D:\;& |f(x)-f(y)|\le d(x,y)\le f(x)+f(y)\\
        &\text{ for all }x,y\in D\}.
\end{aligned}
\]
Elements of \(K(D)\) will be called bounded Kat\v{e}tov functions on \(D\).
\end{definition}

The set \(K(D)\) is a compact convex subset of the product cube \([0,1]^D\).
Compactness follows from closedness of the defining inequalities, and convexity
from the convexity of the same inequalities.  Since \(D\) is countable, \(K(D)\) is
metrizable.

Let
\[
        \iota:\U_1\longrightarrow[0,1]^D,
        \qquad
        \iota(z)(x)=d(z,x)
\]
be the distance-coordinate map.  Let \(S'_D(\U_1)\) denote the closed convex hull of this image, with respect to
the selected countable dense coordinate set \(D\).  Sabok's affine
homeomorphism type for \(S'(\U_1)\) does not depend on the chosen countable
dense coordinate set \cite[Proposition~5.11]{Sabok2016}; hence a failure of
the Choquet simplex property for this coordinate model implies the same failure
for \(S'(\U_1)\).

\begin{proposition}[Kat\v{e}tov model]\label{prop:urysohn-model}
For the coordinate set \(D=\QU_1\), one has
\[
        S'_D(\U_1)=K(D)\subseteq[0,1]^D.
\]
More precisely,
\[
        \cl_{[0,1]^D}\iota(\U_1)=K(D),
\]
and therefore the closed convex hull of \(\iota(\U_1)\) is \(K(D)\).
\end{proposition}

\begin{proof}
For every \(z\in\U_1\), the function \(x\mapsto d(z,x)\) satisfies
\[
        |d(z,x)-d(z,y)|\le d(x,y)\le d(z,x)+d(z,y),
\]
so \(\iota(\U_1)\subseteq K(D)\).  Since \(K(D)\) is closed,
\[
        \cl\iota(\U_1)\subseteq K(D).
\]

Conversely, let \(f\in K(D)\), let \(F\subset D\) be finite, and let
\(\varepsilon>0\).  Put
\[
\begin{aligned}
        K(F)=\{p\in[0,1]^F:\;& |p(x)-p(y)|\le d(x,y)\le p(x)+p(y)\\
        &\text{ for all }x,y\in F\}.
\end{aligned}
\]
This is a rational polytope, because all distances in \(D\) are rational; as a
bounded rational polytope, it is the convex hull of rational vertices, so its
rational points are dense.  Since \(f|F\in K(F)\), choose
\(p\in K(F)\cap(\mathbb Q\cap[0,1])^F\) such that
\[
        |p(x)-f(x)|<\varepsilon\qquad(x\in F).
\]
By the rational one-point extension property of \(D\), there exists
\(z\in D\subseteq\U_1\) such that
\[
        d(z,x)=p(x)\qquad(x\in F).
\]
Thus every basic product neighborhood of \(f\) meets \(\iota(D)\), and hence
meets \(\iota(\U_1)\).  Therefore \(f\in\cl\iota(\U_1)\), proving
\(\cl\iota(\U_1)=K(D)\).

Since \(\iota(\U_1)\subseteq K(D)\) and \(K(D)\) is convex and closed, the
closed convex hull of \(\iota(\U_1)\) is contained in \(K(D)\).  The opposite
inclusion follows from \(K(D)=\cl\iota(\U_1)\subseteq\cl\,\conv(\iota(\U_1))\).
Hence the closed convex hull is exactly \(K(D)\).  This proves
\(S'_D(\U_1)=K(D)\) for the fixed coordinate set \(D=\QU_1\); by the
coordinate-set invariance recalled above, this computes the affine
homeomorphism type of \(S'(\U_1)\).
\end{proof}

We shall use several elementary consequences of the defining inequalities.
They are recorded once in a form adapted to extremality arguments.

\begin{lemma}[Tight inequalities]\label{lem:tight}
Let \(u,v\in K(D)\), let \(e=(u+v)/2\), and set \(h=(u-v)/2\).  For
\(x,y\in D\), the following implications hold:
\begin{enumerate}
\item[\textup{(i)}] If \(e(x)=1\), then \(h(x)=0\).
\item[\textup{(ii)}] If \(e(x)+e(y)=d(x,y)\), then \(h(x)+h(y)=0\).
\item[\textup{(iii)}] If \(e(x)-e(y)=d(x,y)\), then \(h(x)-h(y)=0\).
\item[\textup{(iv)}] If \(e(y)-e(x)=d(x,y)\), then \(h(y)-h(x)=0\).
\end{enumerate}
\end{lemma}

\begin{proof}
If \(e(x)=1\), then \(u(x),v(x)\le1\) and \((u(x)+v(x))/2=1\), whence
\(u(x)=v(x)=1\).

If \(e(x)+e(y)=d(x,y)\), then
\[
        u(x)+u(y)\ge d(x,y),\qquad
        v(x)+v(y)\ge d(x,y),
\]
and the average of the two left-hand sides is \(d(x,y)\).  Both inequalities
are equalities, so \(u(x)+u(y)=v(x)+v(y)\), equivalently \(h(x)+h(y)=0\).

If \(e(x)-e(y)=d(x,y)\), then
\[
        u(x)-u(y)\le d(x,y),\qquad
        v(x)-v(y)\le d(x,y),
\]
and the average of the two left-hand sides is \(d(x,y)\).  Again both
inequalities are equalities, giving \(h(x)-h(y)=0\).  The last assertion is
symmetric.
\end{proof}

\section{Two constant extreme points}
\label{sec:constants}

The first two extreme points are constants.  The constant \(1\) is extreme for
the trivial reason that all coordinates are maximized.  The constant \(1/2\)
is more informative: extremality is forced by equilateral triangles of side
one through every point of \(D\).

\begin{lemma}\label{lem:equilateral-through-point}
In the rational Urysohn sphere \(D=\QU_1\), for every \(x\in D\), there are
\(y,z\in D\) such that
\[
        d(x,y)=d(y,z)=d(z,x)=1.
\]
\end{lemma}

\begin{proof}
Apply the rational one-point extension property first to \(\{x\}\) with value
\(1\), obtaining \(y\in D\) with \(d(x,y)=1\).  Apply it again to
\(\{x,y\}\) with both prescribed values equal to \(1\); the only nontrivial
Kat\v{e}tov inequalities are \(0\le d(x,y)=1\le2\).  The resulting point
\(z\) satisfies \(d(z,x)=d(z,y)=1\), and hence \(x,y,z\) form a side-one
equilateral triangle.
\end{proof}

\begin{proposition}\label{prop:constant-extreme}
The two functions
\[
        \one(x)=1,
        \qquad
        \half(x)=\frac12\qquad(x\in D)
\]
are extreme points of \(K(D)\).
\end{proposition}

\begin{proof}
Both functions belong to \(K(D)\).  The function \(\one\) is extreme because if
\(\one=(u+v)/2\) with \(u,v\in K(D)\), then \(u(x)=v(x)=1\) for all \(x\in D\).

Now suppose that \(\half=(u+v)/2\), and set \(h=(u-v)/2\).  Fix \(x\in D\).
By Lemma~\ref{lem:equilateral-through-point}, choose \(y,z\in D\) such that
\[
        d(x,y)=d(y,z)=d(z,x)=1.
\]
Since
\[
        \half(x)+\half(y)
        =
        \half(y)+\half(z)
        =
        \half(z)+\half(x)
        =1,
\]
all three lower Kat\v{e}tov inequalities are tight.  Lemma~\ref{lem:tight} gives
\[
        h(x)+h(y)=0,
        \qquad
        h(y)+h(z)=0,
        \qquad
        h(z)+h(x)=0.
\]
Adding the first and third equations and using the second gives \(2h(x)=0\).
Thus \(h(x)=0\).  Since \(x\) was arbitrary, \(h=0\), and \(u=v=\half\).
\end{proof}

\section{The half-radius pair}
\label{sec:radius}

Fix, once and for all, an equilateral triple
\[
        A=\{a_0,a_1,a_2\}\subset D,
        \qquad
        d(a_i,a_j)=1\quad(i\ne j).
\]
Such a triple exists by the extension property.  Let
\[
        r(x)=r_A(x)=d(x,A)=\min_{0\le i\le2}d(x,a_i).
\]
Since the distance to a set is \(1\)-Lipschitz,
\[
        |r(x)-r(y)|\le d(x,y)\qquad(x,y\in D).
\]
Define
\begin{equation}\label{eq:def-fg}
        f(x)=f_A(x)=\min\{1,\tfrac12+r(x)\},
        \qquad
        g(x)=g_A(x)=\max\{\tfrac12,1-r(x)\}.
\end{equation}

\begin{lemma}\label{lem:fg-in-K}
More generally, let \(D\) be a metric space of diameter at most one, let
\(A\subseteq D\) be nonempty and finite, and put
\[
        r(x)=d(x,A)=\min_{a\in A}d(x,a).
\]
Let \(K(D)\) denote the set of functions \(h\in[0,1]^D\) satisfying
\[
        |h(x)-h(y)|\le d(x,y)\le h(x)+h(y)
        \qquad(x,y\in D),
\]
and define
\[
        f(x)=\min\{1,\tfrac12+r(x)\},
        \qquad
        g(x)=\max\{\tfrac12,1-r(x)\}.
\]
Then \(f\) and \(g\) belong to \(K(D)\), and
\[
        f+g=\frac32\,\one.
\]
\end{lemma}

\begin{proof}
Because \(A\ne\emptyset\) and \(\diam(D)\le1\), one has \(0\le r\le1\).  The
distance-to-a-set function is \(1\)-Lipschitz, since
\(d(x,A)\le d(x,y)+d(y,A)\) and the symmetric inequality hold.  The scalar maps
\(t\mapsto\min\{1,1/2+t\}\) and \(t\mapsto\max\{1/2,1-t\}\) are also
\(1\)-Lipschitz on \([0,1]\); hence \(f\) and \(g\) satisfy the upper Kat\v{e}tov
inequalities.  Moreover \(1/2\le f,g\le1\), so
\(f(x)+f(y)\ge1\ge d(x,y)\) and \(g(x)+g(y)\ge1\ge d(x,y)\).  Thus
\(f,g\in K(D)\).  Finally, if \(r(x)\le1/2\), then
\(f(x)=1/2+r(x)\) and \(g(x)=1-r(x)\), while if \(r(x)\ge1/2\), then
\(f(x)=1\) and \(g(x)=1/2\).  Hence \(f(x)+g(x)=3/2\) for all \(x\).
\end{proof}

We now prove that \(f\) is extreme.  The mechanism is local: on the region
\(\{r<1/2\}\), \(f\) is forced by tight distance inequalities from \(A\); on
the region \(\{r\ge1/2\}\), \(f\) is forced by the coordinate maximum \(f=1\).

\begin{proposition}\label{prop:f-extreme}
The function \(f=f_A\) is an extreme point of \(K(D)\).
\end{proposition}

\begin{proof}
Suppose that \(f=(u+v)/2\) with \(u,v\in K(D)\), and put \(h=(u-v)/2\).  We
prove \(h=0\).

First consider the three points of \(A\).  For \(i\ne j\),
\[
        f(a_i)+f(a_j)=\frac12+\frac12=1=d(a_i,a_j).
\]
Hence Lemma~\ref{lem:tight} gives
\[
        h(a_i)+h(a_j)=0\qquad(i\ne j).
\]
These three equations imply \(h(a_0)=h(a_1)=h(a_2)=0\).

Let \(x\in D\).  If \(f(x)=1\), then Lemma~\ref{lem:tight} gives \(h(x)=0\).  If
\(f(x)<1\), then \(r(x)<1/2\).  Choose \(a_i\in A\) with \(d(x,a_i)=r(x)\).
Then \(f(x)-f(a_i)=r(x)=d(x,a_i)\).  Lemma~\ref{lem:tight} gives
\(h(x)=h(a_i)=0\).  Thus \(h=0\), and \(f\) is extreme.
\end{proof}

The proof for \(g\) is subtler.  On \(\{r<1/2\}\), \(g\) is again forced by
tight difference inequalities from \(A\).  On the plateau \(\{r\ge1/2\}\),
where \(g=1/2\), we use equilateral triangles contained in the plateau.

\begin{lemma}[Plateau triangles]\label{lem:plateau-triangles}
Let
\[
        B=\{x\in D:r(x)\ge1/2\}.
\]
For every \(x\in B\), there are \(y,z\in B\) such that
\[
        d(x,y)=d(y,z)=d(z,x)=1.
\]
\end{lemma}

\begin{proof}
Fix \(x\in B\), and put \(s_i=d(x,a_i)\).  Then \(1/2\le s_i\le1\) for all
\(i\).  Prescribe a point \(y\) over \(A\cup\{x\}\) by
\(p_y(x)=1\) and \(p_y(a_i)=1/2\).  The Kat\v{e}tov inequalities are exactly
\(|1-1/2|\le s_i\le1+1/2\) for \((x,a_i)\), and
\(0\le d(a_i,a_j)=1\le1\) for \((a_i,a_j)\).  The extension property gives
\(y\in D\) with \(d(y,x)=1\) and \(d(y,a_i)=1/2\), so \(y\in B\).

Now prescribe \(z\) over \(A\cup\{x,y\}\) by
\(p_z(x)=p_z(y)=1\) and \(p_z(a_i)=1/2\).  The checks for \((x,a_i)\) and
\((a_i,a_j)\) are as above; for \((x,y)\) one has \(0\le1\le2\), and for
\((y,a_i)\) one has \(1/2\le d(y,a_i)=1/2\le3/2\).  Thus the extension
property gives \(z\in D\) with \(d(z,x)=d(z,y)=1\) and \(d(z,a_i)=1/2\).  Hence
\(z\in B\), and \(d(x,y)=d(y,z)=d(z,x)=1\).
\end{proof}

\begin{proposition}\label{prop:g-extreme}
The function \(g=g_A\) is an extreme point of \(K(D)\).
\end{proposition}

\begin{proof}
Suppose that \(g=(u+v)/2\) with \(u,v\in K(D)\), and put \(h=(u-v)/2\).  We
prove \(h=0\).

For every \(i\), \(g(a_i)=1\), so Lemma~\ref{lem:tight} gives
\(h(a_i)=0\).

Let \(x\in D\).  If \(r(x)<1/2\), choose \(a_i\in A\) with
\(d(x,a_i)=r(x)\).  Then \(g(a_i)-g(x)=r(x)=d(a_i,x)\).  Lemma~\ref{lem:tight} gives
\(h(x)=h(a_i)=0\).

It remains to consider \(x\in B=\{r\ge1/2\}\).  By
Lemma~\ref{lem:plateau-triangles}, choose \(y,z\in B\) with
\[
        d(x,y)=d(y,z)=d(z,x)=1.
\]
On \(B\), \(g=1/2\), so the three lower inequalities are tight and
Lemma~\ref{lem:tight} yields
\(h(x)+h(y)=h(y)+h(z)=h(z)+h(x)=0\).  Hence \(h(x)=0\).  Therefore \(h=0\),
and \(g\) is extreme.
\end{proof}

\begin{remark}
The side-one triangle in \(A\) removes the possible alternating-sign ambiguity
in the proof of Proposition~\ref{prop:f-extreme}.  The plateau argument in
Proposition~\ref{prop:g-extreme} uses a different feature: the finite extension property
builds, through every point of \(B\), a side-one equilateral triangle contained
in \(B\).
\end{remark}

\section{Failure of Choquet uniqueness for the Urysohn sphere}
\label{sec:urysohn-non-simplex}

We now combine the four extreme points and verify explicitly that the two
representing measures below are genuinely different.

\begin{theorem}\label{thm:urysohn-main}
The compact convex set \(S'(\mathbb U_1)\) is not a Choquet simplex.  In
particular, \(S'(\mathbb U_1)\) is not affinely homeomorphic to the Poulsen
simplex.
\end{theorem}

\begin{proof}
By Proposition~\ref{prop:urysohn-model}, it is enough to prove the assertion for \(K(D)\).
By Propositions~\ref{prop:constant-extreme}, \ref{prop:f-extreme}, and~\ref{prop:g-extreme}, the four points
\[
        f,
        \qquad
        g,
        \qquad
        \one,
        \qquad
        \half
\]
are extreme points of \(K(D)\).  By Lemma~\ref{lem:fg-in-K},
\[
        f+g=\frac32\,\one.
\]
Since also
\[
        \one+\half=\frac32\,\one,
\]
we have
\[
        \frac12 f+\frac12 g
        =
        \frac12\one+\frac12\half
        =
        \frac34\,\one.
\]
Thus the two probability measures
\[
        \mu=\frac12\delta_f+\frac12\delta_g,
        \qquad
        \nu=\frac12\delta_{\one}+\frac12\delta_{\half}
\]
are both carried by \(\ext K(D)\) and have the same barycenter
\[
        \bary(\mu)=\bary(\nu)=\frac34\,\one.
\]

It remains to check that these two measures are different.  For each
\(a_i\in A\), one has \(r(a_i)=0\), hence
\[
        f(a_i)=\frac12,
        \qquad
        g(a_i)=1.
\]
Apply the extension property over \(A\) with prescribed distances
\(p(a_i)=1\) for \(i=0,1,2\).  This prescription is valid because
\[
        |1-1|=0\le d(a_i,a_j)=1\le2=1+1
        \qquad(i\ne j),
\]
and all prescribed distances are rational and lie in \([0,1]\).  Hence there
is \(b\in D\) such that \(d(b,a_i)=1\) for all \(i\).  Then \(r(b)=1\), and
therefore
\[
        f(b)=1,
        \qquad
        g(b)=\frac12.
\]
These evaluations show that \(f\) is neither \(\one\) nor \(\half\), that
\(g\) is neither \(\one\) nor \(\half\), and that \(f\ne g\).  Also
\(\one\ne\half\).  Thus \(\{f,g\}\ne\{\one,\half\}\), so the atomic probability
measures \(\mu\) and \(\nu\) have different supports and are distinct.

As noted above, \(K(D)\) is compact metrizable: it is a closed subset of the
product cube \([0,1]^D\), and \(D\) is countable.  A metrizable Choquet simplex
has a unique representing probability measure carried by the extreme boundary
for each point; see, for example, \cite{Phelps2001,Alfsen1971,AsimowEllis1980}.
The point \((3/4)\one\) has the two different representing measures \(\mu\) and
\(\nu\) carried by \(\ext K(D)\).  Therefore \(K(D)\) is not a Choquet simplex.

The Poulsen simplex is, by definition, a metrizable Choquet simplex with dense
extreme boundary.  Since \(K(D)\) is not a Choquet simplex, it cannot be
affinely homeomorphic to the Poulsen simplex.
\end{proof}

\section{acknowledgements}
The author used GPT-based large language models to generate and compare possible proof strategies, suggest auxiliary lemmas, identify potential gaps, and assist with preliminary wording. All AI-generated suggestions were manually reviewed, selected, modified, and independently verified by the author. No unverified AI-generated proof or mathematical claim was incorporated into the final manuscript.

The authors used the CSSC framework, available at
\url{https://github.com/anetigone/cssc}, together with GPT-based assistance, to
develop Lean 4 formalizations of selected proof components in this paper. The
resulting Lean 4 development is available at
\url{https://github.com/anetigone/sabok-sprime-obstructions-lean}. The
formalization is intended as a machine-checkable supplement to the mathematical
arguments presented here, and should be understood within the scope described in
the accompanying repository.

\end{document}